\RequirePackage{fix-cm}
\documentclass[smallextended]{svjour3}       
\smartqed  
\usepackage{graphicx}
\usepackage{mathptmx}      
\usepackage{amsmath}
\usepackage{amssymb}
\usepackage{nicefrac}
\usepackage{caption}
\usepackage{subcaption}

\newtheorem{ass}{Assumption}%

\usepackage{xcolor}
\usepackage{hyperref}
\makeatletter
\def\cl@chapter{\@elt {chapter}}
\makeatother
\usepackage[noabbrev,capitalize,nameinlink]{cleveref}

\crefformat{equation}{(#2#1#3)}
\crefrangeformat{equation}{(#3#1#4) to~(#5#2#6)}
\crefmultiformat{equation}{(#2#1#3)}{ and~(#2#1#3)}{, (#2#1#3)}{ and~(#2#1#3)}
\crefname{ass}{Assumption}{Assumptions}
\crefname{example}{Example}{Examples}
\newcommand{\sfrac}[2]{\mbox{\footnotesize$\displaystyle\frac{#1}{#2}$}}

\journalname{}

\begin{document}

\title{Convergence analysis of Schwarz-like methods for degenerate elliptic-parabolic equations
\thanks{
	The work of the first author was supported by the Swedish Research Council, grants 2023--03930 and 2023--04862.
	The work of the second author was supported by the Swedish Research Council, grant  2023--04862.}
}
\titlerunning{Schwarz-like methods for degenerate elliptic-parabolic equations}

\author{Monika Eisenmann        \and
             Eskil Hansen%
}

\institute{
  Monika Eisenmann \at
  Centre for Mathematical Sciences, Lund University, P.O.\ Box 118, 221 00 Lund, Sweden \\
  \email{monika.eisenmann@math.lth.se}
  \and
  Eskil Hansen \at
  Centre for Mathematical Sciences, Lund University, P.O.\ Box 118, 221 00 Lund, Sweden \\
  \email{eskil.hansen@math.lth.se}
}
\date{\today}
\maketitle

\begin{abstract}

Convergence is proven for Schwarz-like methods applied to degenerate elliptic-parabolic equations with a $p$-structure. This family of PDEs, e.g., arises when modelling nonlinear diffusion processes. The Schwarz-like approximation methods are based on decomposing the space-time domain into overlapping subdomains, which enables parallel implementations. The methods are derived by introducing a pseudo-time component and applying time integrators of splitting type, which are time stepped towards infinity. This approach of decomposing the space-time domain is related to Schwarz waveform relaxation methods, but the methods considered here have the advantage that they can be proven to converge when applied to nonlinear parabolic, or even degenerate elliptic-parabolic, PDEs. We prove convergence by deriving a nonlinear framework based on the abstract theory for monotone operators and the existence theory for degenerate elliptic-parabolic equations. 

\end{abstract}

\section{Introduction}\label{sec:intro}

Consider the degenerate elliptic-parabolic problem
\begin{equation}\label{eq:parabolic}
\left\{
\begin{aligned}
\partial_t (\gamma u) - \nabla \cdot \alpha(t,\nabla u) +\beta(t,u) +f(t)&= 0 & &\text{in } \Omega\times (0,T),\\
\alpha(t,\nabla u) \cdot \mathbf{n} &=0 & &\text{on } \partial \Omega\times (0,T),\\
\gamma u(0) &=0  & &\text{in } \Omega,
\end{aligned}
\right.
\end{equation}
where $\Omega\in\mathbb{R}^d$ is a bounded Lipschitz domain with the boundary~$\partial\Omega$, and $\mathbf{n}$ denotes the unit outward normal vector of $\partial\Omega$. Here, $\gamma$ is a nonnegative function and $\alpha(t,\nabla u)$ may vanish for nonzero values of $\nabla u$.
That is, for subsets of $\Omega$, where $\gamma(x)=0$ the equation switches from parabolic to elliptic. Furthermore, in contrast to linear parabolic problems, the degeneracy feature that $\alpha$ may be zero for nonzero arguments describes a diffusion process with finite speed of propagation. A survey of applications involving degenerate elliptic-parabolic equations can be found in~\cite[Chapters~8 and~12]{roubicek}. We consider a homogeneous initial condition for the sake of simplicity. For the non-homogeneous case, see~\cite[Chapter~3]{Show} for details. 

As these equations are both nonlinear and require implicit time discretizations, they are numerically challenging and require large-scale computations. A common approach to facilitate such computations is to employ domain decomposition methods, which enable parallel implementations. In the context of elliptic equations, a domain decomposition method consists of first decomposing the equation’s domain into subdomains. Thereafter, the method solves the elliptic equation on each subdomain and communicates the results via the overlaps, or boundaries, of the adjacent subdomains. The classical application of these methods to parabolic equations is to first apply a time integrator and thereafter apply domain decomposition to the elliptic problems arising at each time step. For a general introduction, we refer to \cite{quarteroni,Widlund}. Domain decomposition methods can also be directly incorporated into time integrators by interpreting the decomposition procedure as a splitting scheme, see, e.g.,~\cite{Eisenmann,henningsson,vabi}.

Combining domain decomposition with time integration allows for parallelization in space, but inherently prevents parallel implementations in time. To remedy this, one can consider Schwarz waveform relaxation (SWR) methods, where the full space-time domain is decomposed. These space-time decomposition methods have been proposed in the contexts of parallel time integrators; surveyed in~\cite{gander25}, space-time finite elements; surveyed in~\cite{steinbach19}, and parabolic problems with a spatial domain given by a union of domains with very different material properties~\cite{japhet20}. There are several studies concerning the convergence and other theoretical aspects of space-time decomposition methods applied to non-degenerate parabolic equations, see, e.g.,~\cite{engstrom24,gander07,kwok21,halp12}.

However, results for SWR methods are lacking from the literature when considering more general degenerate parabolic equations. This absence of numerical results is likely explained as follows. The standard convergence analysis for SWR methods applied to linear equations~ \cite{quarteroni} relies on interpreting the method as a projection procedure, which is not applicable for nonlinear equations. For degenerate elliptic equations, i.e., \cref{eq:parabolic} with $\gamma=0$, one can derive convergence analyses by relying on the coercivity of the degenerate elliptic operator~\cite{engstrom22,tai02}. This property is lost for degenerate parabolic equations, as the operator $u\mapsto\partial_t(\gamma u)$ is just  monotone and not necessarily coercive. Hence, if one is to introduce approximation schemes similar to the SWR methods for degenerate parabolic equations, or even degenerate elliptic-parabolic equations, then a more general framework is needed.

The goal of our study is therefore to introduce a new family of domain decomposition schemes that have similar properties to the SWR methods and prove their convergence when applied to degenerate elliptic-parabolic equations. To achieve this, we will derive the new schemes in~\cref{sec:scheme} and cast them into an abstract Cauchy framework based on monotone operators in~\cref{sec:abs}. The convergence can then be proven in the abstract elliptic framework of~\cite{LionsMercier,Temam}, as done in~\cref{sec:conv}. We derive in~\cref{sec:par} how degenerate elliptic-parabolic equations with a $p$-structure can be incorporated into our abstract Cauchy framework. The analysis presented here is restricted to the continuous case, and the design of a parallel space-time finite element method based on the new Schwarz-like methods will be considered elsewhere. 

\section{Constructing Schwarz-like schemes}\label{sec:scheme}

We will derive a family of approximation schemes that are Schwarz-like in the sense that they decompose the space-time domain $\Omega\times (0,T)$ into overlapping subdomains. To this end, consider a collection of subsets $\{\Omega_\ell\}_{\ell=1}^q$ that satisfies $\cup_{\ell=1}^q\Omega_\ell=\Omega$. Each subset $\Omega_\ell$ is either a Lipschitz domain, or a union of pairwise disjoint Lipschitz domains $\{\Omega_{\ell,j}\}_{j=1}^r$ such that $\cup_{j=1}^r\Omega_{\ell,j}=\Omega_\ell$. Over the subdomain $\{\Omega_\ell\}_{\ell=1}^q$ we introduce a partition of unity $\{\chi_\ell\}_{\ell=1}^q$, where the weights satisfy
\begin{equation*}
\chi_\ell(x)>0\text{ a.e.\  }x\in\Omega_\ell,\quad
\chi_\ell(x)=0\text{ a.e.\  }x\in\Omega\setminus\Omega_\ell,\quad
\text{and}\quad\sum_{\ell=1}^q\chi_\ell=1.
\end{equation*}
Given four partition of unities $(\chi_\ell^\alpha,\chi_\ell^\beta,\chi_\ell^\gamma, \chi_\ell^f)$ over $\{\Omega_\ell\}_{\ell=1}^q$, we can formally decompose our degenerate elliptic-parabolic equation~\cref{eq:parabolic} as
\begin{equation}\label{eq:art}
\begin{aligned}
\partial_t (\gamma u) - \nabla &\cdot \alpha(t,\nabla u) +\beta(t,u) +f(t)\\
&=\sum_{\ell=1}^q \partial_t (\chi_\ell^\gamma\gamma u) - \nabla \cdot \chi_\ell^\alpha\alpha(t,\nabla u) + \chi_\ell^\beta \beta(t,u) +\chi_\ell^ff(t)\\
&=\sum_{\ell=1}^q \mathcal{F}_\ell u =0.
\end{aligned}
\end{equation}
Here, each term $\mathcal{F}_\ell u$ is a function with support in the space-time domain $\Omega_\ell\times (0,T)$. With the formulation~\cref{eq:art}, we can derive approximation schemes in the same way as done in~\cite{Temam} for elliptic equations. That is, we introduce a pseudo-time $\xi\in [0,\infty)$ and the corresponding evolution equation
\begin{equation}\label{eq:art2}
\frac{\mathrm{d}}{\mathrm{d} \xi} v+\sum_{\ell=1}^q\mathcal{F}_\ell v=0,\quad v(0)\text{ given}.
\end{equation}
If the maps $\mathcal{F}_\ell$ fulfill some form of monotonicity property, we obtain that $v(\xi)\to u$ as $\xi\to\infty$. Hence, we can approximate $u$ by applying a standard time integrator of splitting type to~\cref{eq:art2} and then time step towards infinity.

For the special case $q=2$, we can employ the Peaceman--Rachford time integrator in the above approach. This gives an approximation scheme of the form: find $\{u^n_1,u^n_2\}_{n\in\mathbb{N}}$ such that
\begin{equation}\label{eq:PR}
\left\{
\begin{aligned}
(sI+\mathcal{F}_1)u_1^{n+1} &=(sI-\mathcal{F}_2)u_2^{n},\\
(sI+\mathcal{F}_2)u_2^{n+1} &=(sI-\mathcal{F}_1)u_1^{n+1},
\end{aligned}
\right.
\end{equation}
where  $u^0_2$ is an initial guess and $u^n_\ell$ is an approximation of $u$ in $\Omega_\ell\times (0,T)$. Here, the
method parameter $s>0$ can be interpreted as $1/\tau$, with $\tau$ denoting the pseudo-time step. If we apply the same procedure to the Douglas--Rachford method, then the corresponding approximation scheme reads
\begin{equation}\label{eq:DR}
\left\{
\begin{aligned}
(sI+\mathcal{F}_1)u_1^{n+1} &=(sI-\mathcal{F}_2)u_2^{n},\\
(sI+\mathcal{F}_2)u_2^{n+1} &=su_1^{n+1}+\mathcal{F}_2u_2^{n}.
\end{aligned}
\right.
\end{equation}
These methods are suitable in this context as they preserve the equilibrium of the pseudo-time problem.
In the general case $q\geq 2$, we can use the additive splitting method to obtain the approximation scheme
\begin{equation}\label{eq:AS}
\left\{
\begin{aligned}
(sI+\mathcal{F}_{\ell})u_\ell^{n+1} &=su^{n}\quad\text{for }\ell=1,\ldots,q,\\
u^{n+1} &=\sfrac{1}{q}\sum_{\ell=1}^q u_\ell^{n+1},
\end{aligned}
\right.
\end{equation}
where $u^n$ is an approximation of $u$ in $\Omega\times (0,T)$. While this method does not necessarily preserve the equilibrium of the equation, it is straightforward to parallelize the additive splitting method.

One can of course create many  other approximation schemes by choosing other splitting integrators. However, we will focus on~\cref{eq:PR,eq:DR,eq:AS}, as they already have quite different convergence properties and require different tailored convergence proofs. Note that it is of little use to consider higher-order splitting methods, as degenerate elliptic-parabolic equations typically have solutions of low regularity.
\section{Preliminaries}\label{sec:prel}

In the following analysis, for a real Banach space $V$, we denote its dual space by $V^*$. When inserting an element $v \in V$ into $u\in V^*$, we use the notation for the dual pairing $u(v) = \langle u, v \rangle_{V^*\times V}$. In the case that there exists a Hilbert space $H$ such that $V$ is densely embedded into $H$ then we have a Gelfand triplet setting $V \hookrightarrow H \hookrightarrow V^*$ and $ \langle u, v \rangle_{V^*\times V}=(u,v)_H$ for $u\in H, v\in  V$. 

Let $G\colon V\to V^*$ be a possibly nonlinear operator. We call $G$ \emph{bounded} if it maps bounded sets in $V$ into bounded sets in $V^*$. If there exists a function $k \colon V \to [0,\infty)$ such that 
\begin{equation*}
		\langle G u - G v, u - v \rangle_{V^* \times V}
		\geq k (u - v)  \quad \text{for all }u,v \in V,
\end{equation*}
then the operator $G$ is referred to as \emph{$k$-monotone}. If $k \equiv0$, then $G$ is just \emph{monotone}. The operator $G$ is \emph{hemicontinuous} if $\varepsilon \mapsto \langle G( u + \varepsilon v), w \rangle_{V^*\times V}$ is continuous for $\varepsilon\in[0,1]$ and all $u,v,w \in V$. The operator $G$ is \emph{coercive} if $\langle Gu , u \rangle_{V^* \times V}\to\infty$ as $\|u\|_V\to\infty$. We call $G$ \emph{symmetric} if $\langle G u, v \rangle_{V^* \times V}= \langle G v, u \rangle_{V^* \times V}$ for all $u,v \in V$.

Additionally, for some statements, we consider $G$ as a possibly unbounded operator on a Hilbert space $H$. In this setting, an operator $G \colon D(G)\subseteq H\to H$, is called \emph{accretive} if
\begin{equation*}
	(G u - G v, u - v)_H \geq 0 \quad \text{for all } u, v \in D(G).
\end{equation*}
Moreover, it is called \emph{maximal} if $R\bigl(s I+G)=\mathcal{H}$ for all $s>0$.

Throughout the paper, $c$ and $C$ will denote generic positive constants.

\section{Abstract Cauchy framework}\label{sec:abs}

Let $H$ be a Hilbert space and $V$ be a separable, reflexive Banach space such that $V$ is densely embedded into $H$. For a given $p\in [2,\infty)$  consider the induced spaces
\begin{equation*}
\mathcal{H}=L^2(0,T; H)\quad\text{and}\quad\mathcal{V}=L^p(0,T; V).
\end{equation*}
By~\cite[Chapter~II.2]{Kufner}, we then have the identifications $\mathcal{V}^*\cong L^{p/(p-1)}(0,T; V^*)$ and
\begin{equation*}
\langle u,v\rangle_{\mathcal{V}^*\times \mathcal{V}}=\int_0^T\langle u(t),v(t)\rangle_{V^*\times V}\,\mathrm{d}t,
\end{equation*}
as well as the Gelfand triplet $\mathcal{V}\hookrightarrow\mathcal{H}\hookrightarrow\mathcal{V}^*$. Furthermore, consider a family of operators $\{A(t)\}_{t\in(0,T)}$, where $A(t)\colon V\to V^*$ is not necessarily linear, and the single bounded linear operator $M \colon H\to H$. The induced operators
$\mathcal{A}\colon \mathcal{V}\to\mathcal{V}^*$ and $\mathcal{M}\colon \mathcal{H}\to\mathcal{H}$ are then given by
\begin{equation*}
(\mathcal{A} u)(t)=A(t)u(t)\quad\text{and}\quad(\mathcal{M} u)(t)=Mu(t)\quad\text{for a.e.\ } t\in (0,T),
\end{equation*}
respectively. The main tool to analyze degenerate elliptic-parabolic equations, is to observe the following. The translation semigroup $\{S(\tau)\}$ on~$\mathcal{V}^*$ defined as
\begin{equation}\label{eq:sg}
S(\tau)u(t)=
\begin{cases}
	0&\text{ for }0\leq t\leq \tau,\\
	u(t-\tau)&\text{ for }\tau<t\leq T,
\end{cases}
\end{equation}
is generated by $-\mathrm{d}/\mathrm{d} t \colon D(\mathrm{d}/\mathrm{d} t)\subseteq \mathcal{V}^*\to\mathcal{V}^*$, where
\begin{equation*}
\begin{aligned}
D(\mathrm{d}/\mathrm{d} t)&=\{u\in \mathcal{V}^*: \lim_{\tau\to 0^+}\sfrac{1}{\tau}\bigl(I-S(\tau)\bigr)u=\mathrm{d}/\mathrm{d} t u\in\mathcal{V}^*\}\\
&=\{u\in W^{1,p/(p-1)}(0,T; V^*): u(0)=0\}.
\end{aligned}
\end{equation*}
Compare with~\cite[Proposition~5.1]{Show}. Note that $W^{1,p/(p-1)}(0,T; V^*)\hookrightarrow C([0,T];V^*)$, by~\cite[Lemma 7.1]{roubicek}, i.e., the pointwise evaluation in $D(\mathrm{d}/\mathrm{d} t)$ is well defined. We also define the space
\begin{equation*}
\mathcal{W}=\{u\in\mathcal{V}: \mathcal{M}u\in D(\mathrm{d}/\mathrm{d} t)\}.
\end{equation*}

For a given $f\in\mathcal{V}^*$, one can now consider the nonlinear Cauchy problem of finding $u\in \mathcal{W}$ such that
\begin{equation}\label{eq:abstract}
\mathcal{F}u=(\mathrm{d}/\mathrm{d} t\mathcal{M}+\mathcal{A})u+f=0\quad\text{in }\mathcal{V}^*,
\end{equation}
or equivalently, finding a solution to
\begin{equation}\label{eq:weak}
-\int_0^T\bigl(Mu(t),\partial_t v(t)\bigr)_H\,\mathrm{d}t+\langle \mathcal{A}u+f,v\rangle_{\mathcal{V}^*\times \mathcal{V}}=0
\end{equation}
for all $v\in W^{1,p}(0,T;V)$ with $v(T)=0$. See~\cite[Chapter~3]{Show} for details on the equivalent formulations of Cauchy problems.
\begin{definition}\label{def:proper}
Consider the spaces $V,H$, which induce $\mathcal{H},\mathcal{V},\mathcal{W}$, together with the operators $M\colon H\to H$, $\mathcal{A}\colon \mathcal{V}\to\mathcal{V}^*$. If
\begin{enumerate}
\item $V$ is a separable, reflexive Banach space that is densely embedded in the Hilbert space $H$;
\item $M$ is linear, bounded, monotone, and symmetric;
\item $\mathcal{A}$ is bounded, $k$-monotone, hemicontinuous, and coercive,
\end{enumerate}
then the problem set $(V,H,M,\mathcal{A})$ is \emph{proper}.
\end{definition}
\begin{lemma}\label{lem:accretive}
If $(V,H,M,\mathcal{A})$ is proper, then
\begin{equation*}
\langle \mathrm{d}/\mathrm{d}t \mathcal{M} u, u\rangle_{\mathcal{V}^*\times\mathcal{V}}\geq 0\quad\text{for all }u\in\mathcal{W}.
\end{equation*}
\end{lemma}
\begin{proof}
Let $u\in\mathcal{W}$ and observe that
\begin{equation}\label{eq:timederivative}
\langle \mathrm{d}/\mathrm{d}t \mathcal{M} u,  u\rangle_{\mathcal{V}^*\times\mathcal{V}}=
\lim_{\tau\to 0^+}\sfrac{1}{\tau}\langle \bigl(I-S(\tau)\bigr)\mathcal{M} u,  u\rangle_{\mathcal{V}^*\times\mathcal{V}}.
\end{equation}
Furthermore, as $M$ is monotone and symmetric, the bilinear form $\langle M\cdot,\cdot\rangle$ satisfies the Cauchy-Schwarz inequality
\begin{equation*}
|\langle Mu,v\rangle_{V^*\times V}|\leq \langle M u,u\rangle^{1/2}_{V^*\times V} \langle Mv,v\rangle^{1/2}_{V^*\times V}
\quad\text{ for all }u,v\in V.
\end{equation*}
This observation implies that for every $u\in\mathcal{W}$, we have
\begin{equation}\label{eq:accretive}
\begin{aligned}
\langle S(\tau)\mathcal{M} u,  &u\rangle_{\mathcal{V}^*\times\mathcal{V}} =\int_\tau^T \langle M u(t-\tau),u(t)\rangle_{V^*\times V}\,\mathrm{d}t\\
&\leq \int_\tau^T \langle M u(t-\tau),u(t-\tau)\rangle^{1/2}_{V^*\times V} \langle M u(t),u(t)\rangle^{1/2}_{V^*\times V}\,\mathrm{d}t\\
&\leq \bigl(\int_0^{T-\tau} \langle M u(t),u(t)\rangle_{V^*\times V}\,\mathrm{d}t\bigr)^{1/2}\bigl(\int_\tau^T \langle M u(t),u(t)\rangle_{V^*\times V}\,\mathrm{d}t\bigr)^{1/2}\\
&\leq \langle \mathcal{M} u,  u\rangle_{\mathcal{V}^*\times\mathcal{V}}.
\end{aligned}
\end{equation}
Combining this with~\cref{eq:timederivative} gives that sought after bound.\qed
\end{proof}
As $-\mathrm{d}/\mathrm{d}t$ generates a contraction semigroup on $\mathcal{V}^*$ satisfying~\cref{eq:accretive}, compare \cite[Chapter~III.5]{Show}, one has for every proper set $(V,H,M,\mathcal{A})$, every $f \in \mathcal{V}^*$ in~\eqref{eq:abstract}, and every $g\in \mathcal{V}^*$ that the equation
\begin{equation}\label{eq:sol}
\mathcal{F}u=g\quad\text{in }\mathcal{V}^*
\end{equation}
has a unique solution $u\in \mathcal{W}$. This is one of the main results from the existence theory for degenerate elliptic-parabolic equations, which is due to~\cite{Bardos}. For an English version see~\cite[Proposition~III.6.2]{Show}. Next, consider the domain
\begin{equation*}
D(\mathcal{F})=\{u\in \mathcal{W}: \mathcal{F} u\in\mathcal{H}\}.
\end{equation*}
together with the restricted operator $\mathcal{F}\colon D(\mathcal{F})\subseteq\mathcal{H}\to\mathcal{H}$, which may be unbounded in $\mathcal{H}$.
For any $s>0$, the problem set $(V,H,M,sI+\mathcal{A})$ is also proper. The observation $\mathcal{H}\subset \mathcal{V}^*$ together with the existence result~\cref{eq:sol} then implies that the restriction of $\mathcal{F}$ is maximal. Furthermore, by \cref{lem:accretive} and the $k$-monotonicity of~$\mathcal{A}$, the restriction of $\mathcal{F}$ is accretive. As a direct consequence, the resolvent
\begin{equation*}
(sI+\mathcal{F})^{-1}\colon \mathcal{H}\to D(\mathcal{F})\subseteq\mathcal{H},
\end{equation*}
is a well-defined, nonexpansive operator for every $s>0$. In order to decompose \cref{eq:abstract}, we assume the following.
\begin{ass}\label{ass:abstract1}
The sets $(V,H,M,\mathcal{A})$ and $(V_\ell,H_\ell,M_\ell,\mathcal{A}_\ell)$, for $\ell=1,\ldots,q$, and the linear, bounded operators $E_{\ell} \colon H_{\ell} \to H$, $R_{\ell} \colon H \to H_{\ell} $ fulfill
\begin{enumerate}
\item  all problem sets are proper;
\item $R_{\ell}E_{\ell}=I$ on $H_\ell$ and $(E_{\ell} u, v)_{H} = (u, R_{\ell} v)_{H_{\ell}}$ for all $u\in H_{\ell},v \in H$;\
\item $Mu=\sum_{\ell=1}^qE_{\ell}M_\ell R_{\ell}u$ for all $u\in H$.
\end{enumerate}
\end{ass}
The induced operators $\mathcal{E_\ell}\colon\mathcal{H}_\ell\to \mathcal{H} $ and $\mathcal{R_\ell}\colon\mathcal{H}\to\mathcal{H}_\ell$, given by
\begin{equation*}
\quad(\mathcal{E}_\ell u)(t)=E_\ell u(t)\quad\text{and}\quad(\mathcal{R}_\ell u)(t)=R_\ell u(t)\quad\text{for a.e.\ } t\in (0,T),
\end{equation*}
respectively, are then linear and bounded. The second and third statements of~\cref{ass:abstract1} also hold for the pairs $\mathcal{E}_\ell,\mathcal{R}_\ell$ in combination with $\mathcal{M}$ and $\mathcal{M}_{\ell}$.

For given functionals $f_\ell\in\mathcal{V}_\ell^*$, $\ell=1,\ldots,q$, we can define the domains
\begin{equation*}
D(\mathcal{F}_\ell)=\{u\in\mathcal{H}: \mathcal{R}_\ell u\in\mathcal{W}_\ell\text{ and } (\mathrm{d}/\mathrm{d} t\mathcal{M}_{\ell}+ \mathcal{A}_{\ell})\mathcal{R}_\ell u+ f_{\ell}\in\mathcal{H}_\ell\}
\end{equation*}
together with the operators $\mathcal{F}_{\ell}\colon D(\mathcal{F}_\ell)\subseteq\mathcal{H}\to \mathcal{H}$, given by
\begin{equation*}
\mathcal{F}_{\ell}u = \mathcal{E}_\ell \bigl((\mathrm{d}/\mathrm{d} t\mathcal{M}_{\ell}+ \mathcal{A}_{\ell})\mathcal{R}_\ell u+ f_{\ell}\bigr) \quad\text{for } u \in D(\mathcal{F}_\ell).
\end{equation*}
\begin{lemma}
If \cref{ass:abstract1} holds then the operator $\mathcal{F}_{\ell}$ is maximal accretive on $\mathcal{H}$.
\end{lemma}
\begin{proof}
By assumption, the set $(\mathcal{V}_\ell,\mathcal{H}_\ell,\mathcal{M}_\ell, sI+\mathcal{A}_\ell)$ is proper and the operator
\begin{equation}\label{eq:intermediatop}
u_\ell \mapsto (\mathrm{d}/\mathrm{d} t\mathcal{M}_{\ell}+\mathcal{A}_{\ell})u_\ell+ f_{\ell}
\end{equation}
is maximal accretive on $\mathcal{H}_\ell$ with domain $\{u_\ell\in \mathcal{W}_\ell: (\mathrm{d}/\mathrm{d} t\mathcal{M}_{\ell}+ \mathcal{A}_{\ell})u_\ell+ f_{\ell}\in\mathcal{H}_\ell\}$.

Hence, for each $g\in \mathcal{H}$ there exists a unique $u_\ell$ in the domain such that
\begin{equation*}
 (sI+\mathrm{d}/\mathrm{d} t\mathcal{M}_{\ell}+ \mathcal{A}_{\ell})u_\ell+ f_{\ell}=\mathcal{R}_\ell g.
\end{equation*}
The function $u=\mathcal{E}_\ell u_\ell+s^{-1}(I-\mathcal{E}_\ell \mathcal{R}_\ell)g$ is then in $D(\mathcal{F}_\ell)$, as $\mathcal{R}_\ell u=u_\ell$, and
\begin{equation*}
\begin{aligned}
(sI+\mathcal{F}_{\ell})u&=s\mathcal{E}_\ell u_\ell+(I-\mathcal{E}_\ell \mathcal{R}_\ell)g\\
&\quad+\mathcal{E}_\ell\Bigl((\mathrm{d}/\mathrm{d} t\mathcal{M}_{\ell}+\mathcal{A}_{\ell})\bigr(\mathcal{R}_\ell \mathcal{E}_\ell u_\ell+s^{-1}\mathcal{R}_\ell(I-\mathcal{E}_\ell \mathcal{R}_\ell)g\bigl)+ f_{\ell}\Bigr)\\
&=\mathcal{E}_\ell\bigl((sI+\mathrm{d}/\mathrm{d} t\mathcal{M}_{\ell}+\mathcal{A}_{\ell})u_\ell+f_\ell\bigr)+(I-\mathcal{E}_\ell \mathcal{R}_\ell)g=g.
\end{aligned}
\end{equation*}
That is, $\mathcal{F}_{\ell}$ is maximal.

Let $u,v\in D(\mathcal{F}_\ell)$ and set $\mathcal{R}_\ell u=u_\ell, \mathcal{R}_\ell v=v_\ell$. As the operator~\cref{eq:intermediatop} is accretive,
\begin{equation*}
\begin{aligned}
(\mathcal{F}_{\ell}u-\mathcal{F}_{\ell}v,u-v)_{\mathcal{H}} &=
\bigl(\mathcal{E}_\ell (\mathrm{d}/\mathrm{d} t\mathcal{M}_{\ell}+ \mathcal{A}_{\ell})\mathcal{R}_\ell u-\mathcal{E}_\ell (\mathrm{d}/\mathrm{d} t\mathcal{M}_{\ell}+ \mathcal{A}_{\ell})\mathcal{R}_\ell v,u-v\bigr)_{\mathcal{H}}\\
&=\bigl((\mathrm{d}/\mathrm{d} t\mathcal{M}_{\ell}+ \mathcal{A}_{\ell})u_\ell-(\mathrm{d}/\mathrm{d} t\mathcal{M}_{\ell}+ \mathcal{A}_{\ell})v_\ell,u_\ell-v_\ell\bigr)_{\mathcal{H}_\ell}\geq 0.
\end{aligned}
\end{equation*}
Hence, $\mathcal{F}_{\ell}$ is accretive.\qed
\end{proof}
\begin{ass}\label{ass:abstract2}
The spaces $V,V_\ell$ and the operators $R_\ell$, $\ell=1,\ldots,q$, fulfill
\begin{enumerate}
\item $R_\ell$ is a bounded operator from $V$ to $V_\ell$;
\item if $u\in H$ such that $R_\ell u\in V_\ell$ for all $\ell=1,\ldots,q$, then $u\in V$ and the bound $\|u\|^p_V\leq C\sum_{\ell=1}^q\|R_\ell u\|^p_{V_\ell}$ holds.
\end{enumerate}
\end{ass}
It directly follows that~\cref{ass:abstract2} also holds for the induced spaces $\mathcal{V},\mathcal{V}_\ell$ and operators $\mathcal{R}_\ell$.
\begin{corollary}\label{cor:Well}
Let~\cref{ass:abstract1,ass:abstract2} be valid. If $u\in \mathcal{H}$ and $\mathcal{R}_\ell u\in \mathcal{W}_\ell$ for all $\ell=1,\ldots,q$, then $u\in \mathcal{W}$ and
\begin{equation*}
\langle \mathrm{d}/\mathrm{d}t\mathcal{M}u,v\rangle_{\mathcal{V}^*\times\mathcal{V}}
=\sum_{\ell=1}^q \langle \mathrm{d}/\mathrm{d}t\mathcal{M}_\ell\mathcal{R}_\ell u, \mathcal{R}_\ell v\rangle_{\mathcal{V}_\ell^*\times\mathcal{V}_\ell}
\end{equation*}
for all $v\in \mathcal{V}$.
\end{corollary}
\begin{proof}
For a $u$ fulfilling the hypothesis we have, by~\cref{ass:abstract2}, that $u\in\mathcal{V}$ and it remains to show that $\mathcal{M}u\in D(\mathrm{d}/\mathrm{d}t)$. To this end, we observe that $\mathcal{R}_\ell u\in \mathcal{W}_\ell$ and the boundedness of $\mathcal{R}_\ell\colon\mathcal{V}\to\mathcal{V}_\ell$ implies
\begin{equation*}
\langle z,\cdot\rangle_{\mathcal{V}^*\times\mathcal{V}}
=\sum_{\ell=1}^q \langle \mathrm{d}/\mathrm{d}t\mathcal{M}_\ell\mathcal{R}_\ell u, \mathcal{R}_\ell (\cdot)\rangle_{\mathcal{V}_\ell^*\times\mathcal{V}_\ell}
\in \mathcal{V}^*.
\end{equation*}
The definition~\cref{eq:sg} of $\{S(\tau)\}$ gives that
\begin{equation*}
\begin{aligned}
\langle \sfrac{1}{\tau}(I-&S(\tau))\mathcal{M}u,v\rangle_{\mathcal{V}^*\times\mathcal{V}}
=\bigl(\sfrac{1}{\tau}(I-S(\tau))\mathcal{M}u,v\bigr)_{\mathcal{H}}\\
&=\sum_{\ell=1}^q \sfrac{1}{\tau}\int_0^T \bigl(E_\ell M_\ell R_{\ell}u(t),v(t)\bigr)_{H}\,\mathrm{d}t
-\sfrac{1}{\tau}\int_\tau^T \bigl(E_\ell M_\ell R_{\ell}u(t-\tau),v(t)\bigr)_{H}\,\mathrm{d}t\\
&=\sum_{\ell=1}^q  \bigl(\sfrac{1}{\tau}(I-S(\tau))\mathcal{M}_\ell \mathcal{R}_{\ell}u,\mathcal{R}_{\ell}v\bigr)_{\mathcal{H}_\ell}=
\sum_{\ell=1}^q  \langle \sfrac{1}{\tau}(I-S(\tau))\mathcal{M}_\ell\mathcal{R}_\ell u, \mathcal{R}_\ell v\rangle_{\mathcal{V}_\ell^*\times\mathcal{V}_\ell}
\end{aligned}
\end{equation*}
for all $u\in\mathcal{H}, v\in \mathcal{V}$. Consider $z_\ell(\tau)=1/\tau(I-S(\tau))\mathcal{M}_\ell\mathcal{R}_\ell u$ and $z_\ell=\mathrm{d}/\mathrm{d}t\mathcal{M}_\ell\mathcal{R}_\ell u$. As $\mathcal{R}_\ell u\in \mathcal{W}_\ell$, we have $z_\ell(\tau)\to z_\ell$ in $\mathcal{V}_\ell^*$as $\tau\to 0^+$. Hence, it follows that
\begin{equation*}
\|\sfrac{1}{\tau}(I-S(\tau))\mathcal{M}u-z\|_{\mathcal{V}^*}
\leq\sup_{v\in \mathcal{V}\backslash\{0\}}\frac{1}{\|v\|_{\mathcal{V}}}\sum_{\ell=1}^q| \langle z_\ell(\tau)-z_\ell,\mathcal{R}_\ell v\rangle_{\mathcal{V}_\ell^*\times\mathcal{V}_\ell}|\to 0
\end{equation*}
as $\tau\to 0^+$. That is, $\mathcal{M}u\in D(\mathrm{d}/\mathrm{d}t)$ and $\mathrm{d}/\mathrm{d}t\mathcal{M}u=z$.\qed
\end{proof}
Note that~\cref{cor:Well} implies the inclusion $\cap_{\ell=1}^q D(\mathcal{F}_\ell)\subseteq \mathcal{W}$.
\begin{ass}\label{ass:abstract3}
The identity
\begin{equation*}
\begin{aligned}
	\mathcal{F}u=&\mathrm{d}/\mathrm{d}t\mathcal{M} u + \mathcal{A} u + f\\
	=& \sum_{\ell=1}^{q} \mathcal{E}_{\ell} ( \mathrm{d}/\mathrm{d}t\mathcal{M}_\ell\mathcal{R}_\ell u + \mathcal{A}_{\ell} \mathcal{R}_{\ell} u + f_{\ell})
	= \sum_{\ell=1}^q \mathcal{F}_\ell u
\end{aligned}
\end{equation*}
in $\mathcal{H}$ holds for every $u\in\cap_{\ell=1}^q D(\mathcal{F}_\ell)$.
\end{ass}
From~\cref{ass:abstract3} it is clear that $\cap_{\ell=1}^q D(\mathcal{F}_\ell)\subseteq  D(\mathcal{F})$, but the assumption does not imply equality in general. In order to proceed with the analysis, we therefore assume the following additional regularity property.
\begin{ass}\label{ass:regularity}
The solution to $\mathcal{F}u=0$ satisfies $u\in \cap_{\ell=1}^q D(\mathcal{F}_\ell)$.
\end{ass}

\section{Abstract method convergence}\label{sec:conv}

We will combine the abstract Cauchy framework in~\cref{sec:abs} together with the elliptic convergence results derived in~\cite[Proposition~1]{LionsMercier} and~\cite[Theorem~3.1]{Temam}. As these results are central to our analysis, we give the proofs in the current notation.
\begin{theorem}\label{thm:convPR}
Consider the Peaceman--Rachford~\cref{eq:PR} or Douglas--Rachford~\cref{eq:DR} approximation $\{u^n_1,u^n_2\}_{n\in\mathbb{N}}$ of the solution~$u$ to the nonlinear Cauchy problem~$\mathcal{F}u=0$. If~\cref{ass:abstract1,ass:abstract2,ass:abstract3,ass:regularity} hold and $u_2^0\in D(\mathcal{F}_2)$ then
\begin{equation*}
\lim_{n\to\infty}k_1(u^n_1-u)+k_2(u^n_2-u)=0
\end{equation*}
for every method parameter $s>0$.
\end{theorem}
\begin{proof}
We begin with the Peaceman--Rachford case. The regularity $u\in D(\mathcal{F}_1)\cap D(\mathcal{F}_2)$ implies that $\mathcal{F}_1u=-\mathcal{F}_2u$. If $u^n_2\in D(\mathcal{F}_2)$ then
\begin{equation*}
\begin{aligned}
u^{n+1}_1&=(sI+\mathcal{F}_1)^{-1}(sI-\mathcal{F}_2)u^n_2\in D(\mathcal{F}_1)\quad\text{and}\\
u^{n+1}_2&=(sI+\mathcal{F}_2)^{-1}(sI-\mathcal{F}_1)u^{n+1}_1\in D(\mathcal{F}_2).
\end{aligned}
\end{equation*}
As $u_2^0\in D(\mathcal{F}_2)$, we have by induction that $\{u^{n}_1,u^{n}_2\}_{n\in\mathbb{N}}\subset  D(\mathcal{F}_1)\times D(\mathcal{F}_2)$. Let
\begin{equation*}
v^n=(sI+\mathcal{F}_2)u^n_2,\quad v=(sI+\mathcal{F}_2)u,\quad
w^n=(sI-\mathcal{F}_2)u^n_2\quad\text{and}\quad w=(sI-\mathcal{F}_2)u.
\end{equation*}
This notation implies the relations
\begin{equation*}
\begin{gathered}
u =\frac{v+w}{2s},\quad u^n_2=\frac{v^n+w^n}{2s},\quad u^{n+1}_1=\frac{v^{n+1}+w^n}{2s},\\
\mathcal{F}_2u=\frac{v-w}{2}, \quad\mathcal{F}_2 u^n_2=\frac{v^n-w^n}{2},\quad
\mathcal{F}_1u=\frac{w-v}{2}, \quad\mathcal{F}_1u^{n+1}_1=\frac{w^n-v^{n+1}}{2}.
\end{gathered}
\end{equation*}
The accretivity of $\mathcal{F}_\ell$ then gives the two bounds
\begin{equation*}
\begin{aligned}
0\leq (\mathcal{F}_2u^n_2-\mathcal{F}_2 u, u^n_2 -u)_{\mathcal{H}}
&= \sfrac{1}{4s} \bigl((v^n-v)-(w^n-w),(v^n-v)+(w^n-w)\bigr)_{\mathcal{H}} \\
&=\sfrac{1}{4s} \bigl(\|v^n-v\|^2_{\mathcal{H}}-\|w^n-w\|^2_{\mathcal{H}}\bigr)
\end{aligned}
\end{equation*}
and
\begin{equation*}
\begin{aligned}
0&\leq (\mathcal{F}_1u^{n+1}_1-\mathcal{F}_1 u, u^{n+1}_1 -u)_{\mathcal{H}} \\
&= \sfrac{1}{4s} \bigl((w^n-w)-(v^{n+1}-v),(w^n-w)+(v^{n+1}-v)\bigr)_{\mathcal{H}} \\
&=\sfrac{1}{4s} \bigl(\|w^n-w\|^2_{\mathcal{H}}-\|v^{n+1}-v\|^2_{\mathcal{H}}\bigr).
\end{aligned}
\end{equation*}
The two bounds then yield $\|v^{n+1}-v\|^2_{\mathcal{H}}\leq \|w^n-w\|^2_{\mathcal{H}} \leq \|v^n-v\|^2_{\mathcal{H}}$. This shows that the real valued sequence $\{\|v^n-v\|^2_{\mathcal{H}}\}_{n \in \mathbb{N}}$ is monotonously decreasing and bounded from below by zero. Thus, it converges and it follows that
\begin{equation*}
\|v^n-v\|^2_{\mathcal{H}}-\|v^{n+1}-v\|^2_{\mathcal{H}}\to 0\quad\text{as }n\to\infty.
\end{equation*}
This combined with the two bounds above and~\cref{eq:accretive}, implies that
\begin{equation*}
0\leq \sum_{\ell=1}^2 k_\ell (u^n_\ell-u) \leq
\sum_{\ell=1}^2 (\mathcal{F}_\ell u^n_\ell-\mathcal{F}_\ell u, u^n_\ell -u)_{\mathcal{H}}\to 0
\end{equation*}
as $n\to\infty$. The convergence proof for the Douglas--Rachford scheme follows in the same fashion and is therefore omitted.
\qed
\end{proof}

For the additive splitting~\cref{eq:AS}, the parameter $s$ needs to be chosen more carefully, as a function of $N$. An optimal choice of $s$ is parameter dependent, compare~\cite[Theorem~5.1]{Temam} for more details. For the sake of simplicity, we choose $s=C\sqrt{N}$ that fulfills \cite[Remark~3.1]{Temam} and guarantees convergence while keeping the notation compact.

\begin{theorem}\label{thm:convAS}
	Consider the additive splitting~\cref{eq:AS} approximation $\{u^n\}_{n=1}^N$ of the solution~$u$ to the nonlinear Cauchy problem~$\mathcal{F}u=0$. If~\cref{ass:abstract1,ass:abstract2,ass:abstract3,ass:regularity} hold and $k_\ell\geq c\|\cdot\|^2_{\mathcal{H}}$ then it follows that
	\begin{equation*}
		\lim_{N\to\infty}\|u^N-u\|_{\mathcal{H}}=0
	\end{equation*}
	for the parameter choice $s=C\sqrt{N}$.
\end{theorem}
\begin{proof}
	Let $v_\ell^{n+1}=u^{n+1}_\ell-u$ and $v^n=u^n-u$. Then~\cref{eq:AS} gives
	\begin{equation*}
		s(v_\ell^{n+1}-v^n)+(\mathcal{F}_\ell u^{n+1}_\ell-\mathcal{F}_\ell u)=-\mathcal{F}_\ell u\in\mathcal{H}.
	\end{equation*}
	By applying $\langle\cdot, sv_\ell^{n+1}\rangle_{\mathcal{V}_\ell^*\times \mathcal{V}_\ell}$ to the equation above, we obtain
	\begin{equation*}
		s^2(v_\ell^{n+1}-v^n,v_\ell^{n+1})_{\mathcal{H}}+s\langle \mathcal{F}_\ell u^{n+1}_\ell-\mathcal{F}_\ell u, v_\ell^{n+1}\rangle_{\mathcal{V}_\ell^*\times \mathcal{V}_\ell}
		=(-\mathcal{F}_\ell u,sv_\ell^{n+1})_{\mathcal{H}}.
	\end{equation*}
	The first term to the right can be rewritten as
	\begin{equation*}
		\begin{aligned}
			s^2(v_\ell^{n+1}-v^n,v_\ell^{n+1})_{\mathcal{H}}&=
			\sfrac{s^2}{2}(v_\ell^{n+1}-v^n,v_\ell^{n+1})_{\mathcal{H}}+\sfrac{s^2}{2}(v_\ell^{n+1}-v^n,v_\ell^{n+1}-v^n)_{\mathcal{H}}\\
			&\quad+\sfrac{s^2}{2}(v_\ell^{n+1}-v^n,v^n)_{\mathcal{H}}\\
			&=\sfrac{s^2}{2}\|v^{n+1}_\ell\|^2_{\mathcal{H}}+\sfrac{s^2}{2}\|v^{n+1}_\ell-v^n\|^2_{\mathcal{H}}-\sfrac{s^2}{2}\|v^n\|^2_{\mathcal{H}},
		\end{aligned}
	\end{equation*}
	and the bound $k_\ell\geq c\|\cdot\|^2_{\mathcal{H}}$ then yields the inequality
	\begin{equation*}
		s(s+2c)\|v^{n+1}_\ell\|^2_{\mathcal{H}}-s^2\|v^n\|^2_{\mathcal{H}}+s^2\|v^{n+1}_\ell-v^n\|^2_{\mathcal{H}}\leq 2(-\mathcal{F}_\ell u,sv_\ell^{n+1})_{\mathcal{H}}.
	\end{equation*}
	Taking the average of these inequalities, for $\ell=1,\ldots,q$, and noting that
	\begin{equation*}
		\|v^{n+1}\|^2_{\mathcal{H}}\leq \sfrac{1}{q} \sum_{\ell=1}^q\|v^{n+1}_\ell\|^2_{\mathcal{H}},
	\end{equation*}
	by the convexity of $\|\cdot\|_{\mathcal{H}}^2$, implies
	\begin{equation*}
		s(s+2c)\|v^{n+1}\|^2_{\mathcal{H}}-s^2\|v^n\|^2_{\mathcal{H}}+\sfrac{s^2}{q} \sum_{\ell=1}^q\|v^{n+1}_\ell-v^n\|^2_{\mathcal{H}}\leq\sfrac{2}{q}\sum_{\ell=1}^q (-\mathcal{F}_\ell u,sv_\ell^{n+1})_{\mathcal{H}}.
	\end{equation*}
	As $\mathcal{F} u=\sum_{\ell=1}^q\mathcal{F}_\ell u=0$, the right-hand side of the above inequality is bounded by
	\begin{equation*}
		\begin{aligned}
			\sfrac{2}{q}\sum_{\ell=1}^q (-\mathcal{F}_\ell &u,sv_\ell^{n+1})_{\mathcal{H}}=
			\sfrac{2}{q}\sum_{\ell=1}^q \big(\bigl(-\mathcal{F}_\ell u,s(v_\ell^{n+1}-v^n)\bigr)_{\mathcal{H}}+(-\mathcal{F}_\ell u,sv^n)_{\mathcal{H}}\big)\\
			&=\sfrac{2}{q}\sum_{\ell=1}^q \bigl(-\mathcal{F}_\ell u,s(v_\ell^{n+1}-v^n)\bigr)_{\mathcal{H}}
			\leq\sfrac{1}{q}\sum_{\ell=1}^q \big(\|\mathcal{F}_\ell u\|^2_{\mathcal{H}}+s^2\|v_\ell^{n+1}-v^n\|^2_{\mathcal{H}}\big).
		\end{aligned}
	\end{equation*}
	Hence, it follows that
	\begin{equation*}
		s(s+2c)\|v^{n+1}\|^2_{\mathcal{H}}-s^2\|v^n\|^2_{\mathcal{H}}\leq  \sfrac{1}{q}\sum_{\ell=1}^q\|\mathcal{F}_\ell u\|^2_{\mathcal{H}}.
	\end{equation*}
	For the constant $C(\mathcal{F}_\ell)=1/q\sum_{\ell=1}^q\|\mathcal{F}_\ell u\|^2_{\mathcal{H}}$, we then obtain the bound
	\begin{equation*}
		\begin{aligned}
			\|u^N-u\|^2_{\mathcal{H}}&\leq \frac{1}{(1+2c/s)^N}\|u^0-u\|^2_{\mathcal{H}}+C(\mathcal{F}_\ell)\frac{1}{s^2}\sum^N_{n=1}\frac{1}{(1+2c/s)^n}\\
			&=\frac{1}{(1+2c/s)^N}\|u^0-u\|^2_{\mathcal{H}}+C(\mathcal{F}_\ell)\frac{1}{2cs}\bigl(1-\frac{1}{(1+2c/s)^N}\bigr).
		\end{aligned}
	\end{equation*}
	The parameter choice $s=C\sqrt{N}$ then gives the sought after convergence in $\mathcal{H}$, as $1/(1+C/\sqrt{N})^{N}\to 0$ as $N\to\infty$.
		\qed
\end{proof}

\section{Degenerate elliptic-parabolic equations with $p$-structures}\label{sec:par}

Multiplying the degenerate elliptic-parabolic equation~\cref{eq:parabolic} with a sufficiently regular test function $v$, with $v(T)=0$, formally gives the weak form
\begin{equation*}
\int_0^T\int_\Omega -\gamma u\partial_t v+\alpha(t,\nabla u)\cdot \nabla v+\beta(t,u)v+f(t)v\,\mathrm{d}x\mathrm{d}t=0.
\end{equation*}
Comparing with~\cref{eq:weak} motivates the choices
\begin{equation*}
H=L^2(\Omega),\quad V=W^{1,p}(\Omega),\quad \text{and}\quad (Mu)(x)=\gamma(x)u(x),
\end{equation*}
where $p\in[2,\infty)$, $\gamma\in L^\infty(\Omega)$ is nonnegative, together with the operators $A(t)\colon V\to V^*$, $t\in (0,T)$, given by
\begin{equation*}
\langle A(t)u,v\rangle_{V^*\times V}=\int_\Omega \alpha(t,\nabla u)\cdot \nabla v+\beta(t,u)v\,\mathrm{d}x
\quad\text{for }u,v\in V.
\end{equation*}
To decompose $H$ we consider $H_\ell=L^2(\Omega_\ell)$, with the usual norm. The connection between $H$ and $H_{\ell}$ is the given by the zero extension operator $E_{\ell} \colon H_{\ell} \to H$ and the restriction operator $R_{\ell} \colon H  \to H_{\ell}$, i.e.,
\begin{equation*}
 (E_{\ell} u_{\ell}) (x) =
 \begin{cases}
	0 & \text{ for }x \in \Omega \setminus\Omega_\ell\\
	u_{\ell}(x)  &\text{ for }x\in\Omega_\ell
 \end{cases}
 \quad
 \text{and}\quad R_{\ell} u = u|_{\Omega_\ell}
\end{equation*}
for all $u_{\ell} \in H_{\ell}$ and $u \in H$. Note that these operators fulfill $(E_{\ell} u_{\ell}, u)_{H} = (u_{\ell},R_{\ell} u)_{H_{\ell}}$.

Next, we will introduce the decomposed operators and source terms. To this end, recall the overlapping subdomains $\{\Omega_\ell\}_{\ell=1}^q$ from~\cref{sec:scheme}, and consider the partition of unities with the weights $\chi_\ell^\alpha=E_\ell a_\ell$,  $\chi_\ell^\beta=E_\ell b_\ell$, and $\chi_\ell^\gamma=E_\ell g_\ell$. Here, we assume
\begin{equation*}
a_\ell\in W_0^{1,\infty}(\Omega_\ell),\quad b_\ell\in \{v\in L^{\infty}(\Omega_\ell): v(x)\geq c\text{ for a.e.\ $x\in\Omega_\ell$}\},\quad\text{and}\quad g_\ell\in L^{\infty}(\Omega_\ell),
\end{equation*}
for $\ell=1,\ldots,q$. A possible choice of~$\chi_\ell^f$ will be presented in~\cref{ex:f}. 

For an arbitrary weight $w_\ell \in L^{\infty}(\Omega_{\ell})$, let $L^p(\Omega_\ell,w_\ell)$ be the set of measurable functions $u$ on $\Omega_\ell$ such that the weighted norm
\begin{equation*}
\|u\|_{L^p(\Omega_\ell,w_\ell)}=\bigl(\int_{\Omega_\ell}w_\ell |u|^p\,\mathrm{d}x\bigr)^{\frac{1}{p}}
\end{equation*}
is finite. This is a separable and reflexive Banach space. If $w_\ell(x)\geq c$ for a.e.\ $x\in\Omega_\ell$ then $L^p(\Omega_{\ell},w_\ell) $ is isometric isomorphic to $L^p(\Omega_{\ell})$. We can then define the spaces
\begin{equation*}
\begin{aligned}
V_\ell =\bigl\{&u \in L^p(\Omega_{\ell},b_\ell) : \text{ there exists $z_j\in L^p(\Omega_\ell,a_\ell), j=1,\ldots, d$, with}\\
& \langle a_\ell(\nabla u)_j,v \rangle_{W^{-1,p}(\Omega_{\ell})\times W_0^{1,p}(\Omega_{\ell})}=-\int_{\Omega_\ell}a_\ell z_jv\,\mathrm{d}x
\quad\text{for all }v\in W_0^{1,p}(\Omega_{\ell})\bigr\}
\end{aligned}
\end{equation*}
for $\ell=1,\ldots,q$. We equip $V_\ell$ with the norm
\begin{equation*}
\|u\|_{V_\ell}=\bigl(\|\nabla u\|^p_{L^p(\Omega_\ell,a_\ell)^d}+\|u\|^p_{L^p(\Omega_\ell,b_\ell)}\bigr)^{1/p}.
\end{equation*}
As for standard weak derivatives, we will not distinguish between the distributional gradient of $u\in V_\ell$ and the associated vector of weighted $L^p$-functions $z=(z_1,\ldots z_d)$. Instead, both will be denoted by~$\nabla u$.
\begin{lemma}\label{lem:Vell}
The spaces $V_\ell, \ell=1,\ldots,q$, are separable, reflexive Banach spaces, which are all densely embedded into $H_{\ell}$.
\end{lemma}
The proof follows as in~\cite[Lemmas~1--3]{Eisenmann} together with the observation that the separability also holds by~\cite[Theorem~1.22]{Adams}.

The operators $M_\ell \colon H_\ell \to H_\ell, \ell=1,\ldots,q$, given by
\begin{equation*}
(M_\ell u)(x)=g_\ell(x)(R_\ell \gamma)(x)u(x)\quad\text{for a.e.\ }x\in\Omega_\ell,
\end{equation*}
are then linear, bounded, monotone, symmetric, and fulfill $\sum_{\ell=1}^q E_{\ell}M_\ell R_{\ell}=M$. The families of operators $A_\ell(t) \colon V_\ell\to V_\ell^*$, $t\in (0,T), \ell=1,\ldots,q$, are defined by
\begin{equation*}
\langle A_\ell(t)u,v\rangle_{V_\ell^*\times V_\ell}=\int_{\Omega_\ell}a_\ell \alpha(t,\nabla u)\cdot \nabla v+b_\ell\beta(t,u)v\,\mathrm{d}x
\quad\text{for }u,v\in V_\ell.
\end{equation*}

We will consider degenerate elliptic-parabolic equations that have a $p$-structure, i.e., the functions $\alpha,\beta$ have the properties below.
\begin{ass}\label{ass:parabolic}
Let the function $d_1\in L^{p/(p-1)}\bigl(\Omega\times(0,T)\bigr)$ be nonnegative and $d_2\in  L^1\bigl(\Omega\times(0,T)\bigr)$. The functions $\alpha\colon\Omega\times (0,T)\times \mathbb{R}^d\to\mathbb{R}^d$ and $\beta\colon\Omega\times (0,T)\times \mathbb{R}\to \mathbb{R}$ satisfy, for a.e.\ $(x,t)\in \Omega \times (0,T)$,
\begin{enumerate}
\item the maps $\alpha(x,t,z)$ and $\beta(x,t,y)$ are measurable in $x,t$, and continuous in $y,z$;
\item for all $y\in  \mathbb{R},z\in  \mathbb{R}^d$ one has the bounds
\begin{equation*}
|\alpha(x,t,z)|\leq C|z|^{p-1}+d_1(x,t)\quad\text{and}\quad|\beta(x,t,y)|\leq C|y|^{p-1}+d_1(x,t);
\end{equation*}
\item for all $y_1,y_2\in\mathbb{R}, z_1,z_2\in\mathbb{R}^d$ it holds that
\begin{equation*}
\begin{aligned}
\bigl(\alpha(x,t,z_1)-&\alpha(x,t,z_2)\bigr)\cdot(z_1-z_2)+\bigl(\beta(x,t,y_1)-\beta(x,t,y_2)\bigr)(y_1-y_2)\\
&\geq c(|z_1-z_2|^p+|y_1-y_2|^p);
\end{aligned}
\end{equation*}
\item for all $y\in\mathbb{R}, z\in\mathbb{R}^d$ one has $\alpha(x,t,z)\cdot z+\beta(x,t,y) y\geq c(|z|^p+|y|^p)-d_2(x,t)$.
\end{enumerate}
\end{ass}
\begin{example}
The standard case that fulfills~\cref{ass:parabolic} is the parabolic $p$-Laplace problem with a nonlinear reaction term, i.e.,
\begin{equation*}
\alpha(x,t,z)=|z|^{p-2}z\quad\text{and}\quad \beta(x,t,y)=|y|^{p-2}y+\lambda y,
\end{equation*}
where $\lambda\geq 0$.
\end{example}
\begin{ass}\label{ass:rhs}
The source term $f\in \mathcal{V}^*$ in~\cref{eq:parabolic} can be decomposed as 
\begin{equation*}
\langle f,v\rangle_{\mathcal{V}^*\times\mathcal{V}} =\sum_{\ell=1}^q\langle f_\ell,\mathcal{R}_\ell v\rangle_{\mathcal{V}_\ell^*\times\mathcal{V}_\ell} 
\end{equation*}
for all $v\in \mathcal{V}$.
\end{ass}
\begin{example}\label{ex:f}
Let $\eta_0\in L^{p/(p-1)}\bigl(\Omega\times (0,T)\bigr)$ and $\eta\in L^{p/(p-1)}\bigl(\Omega\times (0,T)\bigr)^d$. Then the functional $f$ given by
\begin{align*}
	\langle f, v \rangle_{\mathcal{V}^* \times \mathcal{V}}
	= \int_{0}^{T} \int_{\Omega} \eta_0(t) v(t) + \eta(t)\cdot\nabla v(t)\,\mathrm{d}x\mathrm{d}t \quad\text{for }v\in V
\end{align*}
is an element in $\mathcal{V}^*$. We can then decompose $f$ into elements $f_{\ell}$ in $\mathcal{V}_{\ell}^*$ through
\begin{align*}
	\langle f_{\ell}, v \rangle_{\mathcal{V}_{\ell}^* \times \mathcal{V}_{\ell}}= \int_{0}^{T} \int_{\Omega_{\ell}}b_{\ell} R_{\ell} \eta_0(t) v(t) + a_{\ell} R_{\ell} \eta(t)\cdot \nabla v(t)\,\mathrm{d}x\mathrm{d}t \quad\text{for }v \in \mathcal{V}_{\ell}.
\end{align*}
This implies that
\begin{align*}
	\sum_{\ell=1}^q\langle f_\ell,\mathcal{R}_\ell v\rangle_{\mathcal{V}_\ell^*\times\mathcal{V}_\ell}
	&= \sum_{\ell=1}^q\int_{0}^{T} \int_{\Omega_{\ell}} b_{\ell} R_{\ell} \eta_0(t) R_{\ell} v(t) + a_{\ell} R_{\ell} \eta(t)\cdot \nabla R_{\ell} v(t)\,\mathrm{d}x \mathrm{d}t\\
	&= \int_{0}^{T} \int_{\Omega} \sum_{\ell=1}^q (E_\ell b_{\ell})\eta_0(t) v(t) + (E_\ell a_{\ell}) \eta(t)\cdot \nabla v(t)\,\mathrm{d}x \mathrm{d}t\\
	&= \langle f,v\rangle_{\mathcal{V}^*\times\mathcal{V}},
\end{align*}
for all $v\in \mathcal{V}$, i.e.,~\cref{ass:rhs} holds for this family of functionals. 
\end{example}
\begin{lemma}\label{lem:Aell}
If~\cref{ass:parabolic} holds then $A_\ell(t)\colon V_\ell\to V_\ell^*$ satisfies, for a.e.\ $t\in (0,T)$,
\begin{enumerate}
\item $\|A_\ell(t)u\|_{V_\ell^*}\leq C\|u\|_{V_\ell}^{p-1}+d_3(t)$ for all $u\in V_\ell$;
\item $k_\ell$-monotone with $k_\ell u=c\|u\|_{V_\ell}^p$;
\item $A_\ell(t)$ is hemicontinuous;
\item $\langle A_{\ell}(t)u,u\rangle_{V_{\ell}^*\times V_{\ell}}\geq c\|u\|_{V_{\ell}}^p-d_4(t)$ for all $u\in V_{\ell}$;
\item $t\mapsto\langle A_\ell(t) u, v \rangle_{V_\ell^*\times V}$ is measurable on $(0,T)$ for all $u,v\in V_\ell$,
\end{enumerate}
where $d_3\in L^{p/(p-1)}(0,T)$ is nonnegative and $d_4\in L^1(0,T)$. The same properties hold for $A(t)$, with $V_\ell$ replaced by $V$.
\end{lemma}
\begin{proof}
The first assertion is valid, as for all $u,v\in V_\ell$ and a.e.\ $t\in (0,T)$ one has that
\begin{equation*}
\begin{aligned}
&|\langle A_\ell(t)u, v\rangle_{V_\ell^*\times V_\ell}\\
&\leq
\Bigl(C(\int_{\Omega_\ell} a_\ell |\nabla u|^p\,\mathrm{d}x)^{\frac{p-1}{p}}+(\int_{\Omega_\ell} a_\ell d_1(t)^{\frac{p}{p-1}}\,\mathrm{d}x)^{\frac{p-1}{p}}\Bigr)(\int_{\Omega_\ell} a_\ell |\nabla v|^p\,\mathrm{d}x)^{\frac{1}{p}}\\
&\quad+\Bigl(C(\int_{\Omega_\ell} b_\ell |u|^p\,\mathrm{d}x)^{\frac{p-1}{p}}+(\int_{\Omega_\ell} b_\ell d_1(t)^{\frac{p}{p-1}}\,\mathrm{d}x)^{\frac{p-1}{p}}\Bigr)(\int_{\Omega_\ell} b_\ell |v|^p\,\mathrm{d}x)^{\frac{1}{p}}\\
&\leq C(\|u\|_{V_\ell}^{p-1}+\|d_1(t)\|_{L^{p/(p-1)}(\Omega)})\|v\|_{V_\ell}.
\end{aligned}
\end{equation*}
The second assertion holds, as
\begin{equation*}
\begin{aligned}
\langle A_\ell &(t)u-A_\ell(t)v, u-v\rangle_{V_\ell^*\times V_\ell}\\
&=\int_{\Omega_\ell} a_\ell \bigl(\alpha(t,\nabla u)-\alpha(t,\nabla v)\bigr)\cdot \nabla (u-v)+b_\ell \bigl(\beta(t,u)-\beta(t, v)\bigr)(u-v)\,\mathrm{d}x\\
&\geq c\bigl(\| \nabla (u-v)\|_{L^p(\Omega_\ell,a_\ell)^d}^{p}+\| u-v\|_{L^p(\Omega_\ell,b_\ell)}^{p}\bigr)=c\|u-v\|^p_{V_\ell}\quad\text{for all }u,v\in V_\ell.
\end{aligned}
\end{equation*}
To prove the third assertion, consider a sequence $\{\varepsilon_n\}\subset [0,1]$ with the limit $\epsilon$, elements $u,v,w\in V_\ell$, and the function
\begin{equation*}
h(\varepsilon, x,t) = a_\ell(x)\alpha\bigl(x,t,(\nabla u+\varepsilon \nabla v)(x)\bigl)\cdot \nabla w(x)+ b_\ell(x)\beta\bigl(x,t,(u+\varepsilon v)(x)\bigl)w(x).
\end{equation*}
By~\cref{ass:parabolic}, we have $h(\varepsilon_n, x,t)\to h(\varepsilon, x,t)$ for a.e.\ $(x,t)\in\Omega_\ell\times (0,T)$,  and
\begin{equation*}
\begin{aligned}
|h(\varepsilon, x,t)| &\leq Ca_\ell(x)\bigl((|\nabla u(x)|+|\nabla v(x)|)^{p-1}+c|d_1(x,t)|\bigr)|\nabla w(x)|\\
&\quad+ Cb_\ell(x)\bigl((|u(x)|+|v(x)|)^{p-1}+c|d_1(x,t)|\bigr)|w(x)|,
\end{aligned}
\end{equation*}
where the right-hand side is in $L^1(\Omega_\ell)$ for a.e.\ $t$. Hence, the dominated convergence theorem gives
\begin{equation*}
\lim_{n\to\infty}\langle A_\ell(t)(u+\varepsilon_n w),v\rangle_{V_\ell^*\times V_\ell}=\lim_{n\to\infty}\int_{\Omega_\ell} h(\varepsilon_n, x,t)\,\mathrm{dx}
=\langle A_\ell(t)(u+\varepsilon w),v\rangle_{V_\ell^*\times V_\ell},
\end{equation*}
which implies that $A_\ell(t)$ is hemicontinuous for a.e.\ $t$. The fourth assertion holds, as
\begin{align*}
	\langle A_\ell(t)u, u \rangle_{V_\ell^*\times V_\ell}
	&=\int_{\Omega_\ell} a_\ell\alpha(t,\nabla u)\cdot \nabla u +b_\ell \beta(t,u)u \,\mathrm{d}x\\
	&\geq \int_{\Omega_\ell} c (a_\ell |\nabla u|^p + b_\ell |u|^p) - d_2(x,t)\,\mathrm{d}x
	\geq c \| u\|_{V_{\ell} }^{p} - \| d_2(\cdot,t)\|_{L^1(\Omega_\ell)}
\end{align*}
for all $u \in V_\ell$. The final assertion holds, by the measurability of $\alpha,\beta$ and the argument in~\cite[Section~30.4]{Zeidler}. Repeating the same proof, with the weights $a_\ell,b_\ell$ replaced by~1, gives the same properties for $A(t) \colon V\to V^*$. \qed
\end{proof}
\cref{lem:Aell} together with~\cite[Section~30.3b]{Zeidler} gives that the induced operators $\mathcal{A_\ell}\colon \mathcal{V}_\ell\to\mathcal{V}_\ell^*, \ell=1,\ldots,q$, are all bounded, hemicontinuous, $k_\ell$-monotone, and coercive. Here, the function $k_\ell$ has the form
\begin{equation*}
k_\ell u= c \| u\|_{\mathcal{V}_{\ell} }^{p}.
\end{equation*}
The same properties also hold for $\mathcal{A}\colon\mathcal{V}\to\mathcal{V}^*$. That is, in the context of degenerate elliptic-parabolic equations, the problem sets $(V,H,M,\mathcal{A})$ and $(V_\ell,H_\ell,M_\ell,\mathcal{A}_\ell)$ are all proper. Hence, \cref{ass:parabolic} implies \cref{ass:abstract1}.

Validating~\cref{ass:abstract2} can be done as follows. Since $L^p(\Omega_{\ell}) \subseteq L^p(\Omega_{\ell}, w_{\ell})$ for every weight function $w_{\ell}$ that we have considered, we find that $W^{1,p}(\Omega_\ell)\subset V_\ell$. For $u\in V$, it therefore follows that $R_\ell $ is bounded from $V$ to $V_\ell$ as
\begin{equation*}
	\begin{aligned}
		\| R_\ell u\|^p_{V_{\ell}}
		&=\bigr\| \nabla R_\ell u\bigr\|^p_{L^p(\Omega_{\ell},a_\ell)^d} + \bigr\| R_\ell u \bigr\|^p_{L^p(\Omega_{\ell},  b_\ell)}
		\leq C \big(\bigr\| \nabla R_\ell u\bigr\|^p_{L^p(\Omega_{\ell})^d} + \bigr\| R_\ell u \bigr\|^p_{L^p(\Omega_{\ell})}\big)\\
		&= C \big(\bigr\| E_{\ell} \nabla R_\ell u\bigr\|^p_{L^p(\Omega)^d} + \bigr\| E_{\ell} R_\ell u \bigr\|^p_{L^p(\Omega)}\big)
		\leq C \|u\|^p_{V}.
	\end{aligned}
\end{equation*}
The second part of~\cref{ass:abstract2} follows from
\begin{align*}
	\|u\|_V^p 
	&=\bigr\| \nabla u\bigr\|^p_{L^p(\Omega)^d} + \bigr\| u \bigr\|^p_{L^p(\Omega)} \\
	&\leq \sum_{\ell=1}^{q} \bigr\| \nabla R_{\ell} u\bigr\|^p_{L^p(\Omega_{\ell}, a_{\ell})^d} + \sum_{\ell=1}^{q} \bigr\| R_{\ell} u \bigr\|^p_{L^p(\Omega_{\ell}, b_{\ell})} 
	= \sum_{\ell=1}^{q} \bigr\| R_{\ell} u\bigr\|^p_{V_{\ell}}.
\end{align*}
To prove that \cref{ass:parabolic,ass:rhs} implies \cref{ass:abstract3} first observe that
\begin{equation*}
\begin{aligned}
\sum_{\ell=1}^q &\langle \mathcal{A}_\ell\mathcal{R}_\ell u, \mathcal{R}_\ell v\rangle_{\mathcal{V}_\ell^*\times\mathcal{V}_\ell}\\
&=\sum_{\ell=1}^q\int_0^T\int_{\Omega_\ell} a_\ell\alpha\bigr(t,\nabla R_\ell u(t)\bigr)\cdot\nabla R_\ell v(t)+b_\ell\beta\bigr(t,R_\ell u(t)\bigr)R_\ell v(t)\,\mathrm{d}x\mathrm{d}t\\
&=\int_0^T\int_{\Omega} \sum_{\ell=1}^q (E_\ell a_\ell)\alpha\bigr(t,\nabla u(t)\bigr)\cdot\nabla v(t)+(E_\ell b_\ell)\beta\bigr(t,u(t)\bigr)v(t)\,\mathrm{d}x\mathrm{d}t\\
&=\langle \mathcal{A}u, v\rangle_{\mathcal{V}^*\times\mathcal{V}}
\end{aligned}
\end{equation*}
for every $u,v\in \mathcal{V}$. Hence, by ~\cref{cor:Well}, one has the equality
\begin{equation*}
\begin{aligned}
\bigl(\sum_{\ell=1}^q\mathcal{F}_\ell &u, v\bigr)_{\mathcal{H}}=
\sum_{\ell=1}^q\bigl((\mathrm{d}/\mathrm{d}t\mathcal{M}_\ell+\mathcal{A}_\ell)\mathcal{R}_\ell u+f_\ell, \mathcal{R}_\ell v\bigr)_{\mathcal{H}_\ell}\\
&=\sum_{\ell=1}^q \langle \mathrm{d}/\mathrm{d}t\mathcal{M}_\ell\mathcal{R}_\ell u, \mathcal{R}_\ell v\rangle_{\mathcal{V}_\ell^*\times\mathcal{V}_\ell}+
\langle\mathcal{A}_\ell\mathcal{R}_\ell u, \mathcal{R}_\ell v\rangle_{\mathcal{V}_\ell^*\times\mathcal{V}_\ell}+
\langle f_\ell, \mathcal{R}_\ell v\rangle_{\mathcal{V}_\ell^*\times\mathcal{V}_\ell}\\
&= \langle (\mathrm{d}/\mathrm{d}t\mathcal{M}+\mathcal{A})u+f, v\rangle_{\mathcal{V}^*\times\mathcal{V}}= \langle\mathcal{F}u, v\rangle_{\mathcal{V}^*\times\mathcal{V}}
\end{aligned}
\end{equation*}
for every $u\in \cap_{\ell=1}^q D(\mathcal{F}_\ell),v\in \mathcal{V}$. As $\mathcal{V}$ is dense in $\mathcal{H}$, we have $u\in D(\mathcal{F})$ and $\sum_{\ell=1}^q\mathcal{F}_\ell u=\mathcal{F}u$ in $\mathcal{H}$, i.e., \cref{ass:abstract3} holds.

With this setting~\cref{thm:convPR} translates into the convergence result below.
\begin{corollary}
Consider the Peaceman--Rachford~\cref{eq:PR} or Douglas--Rachford~\cref{eq:DR} approximation $\{u^n_1,u^n_2\}_{n\in\mathbb{N}}$ of the solution~$u$ to the degenerate elliptic-parabolic problem $\mathcal{F}u=0$. If~\cref{ass:parabolic,ass:regularity,ass:rhs} hold and $u_2^0\in D(\mathcal{F}_2)$ then
\begin{equation*}
\lim_{n\to\infty}\,\sum_{\ell=1}^2 c\|\nabla (u^n_\ell-u)\|_{L^p(0,T;\, L^p(\Omega_\ell,a_\ell)^d)}^p+c\|u^n_\ell-u\|_{L^p(0,T;\, L^p(\Omega_\ell,b_\ell))}^p=0
\end{equation*}
for every method parameter $s>0$.
\end{corollary}

For the additive splitting~\cref{eq:AS}, one needs a slight modification, as~\cref{ass:parabolic} does not imply the condition $k_\ell\geq c\|\cdot\|^2_{\mathcal{H}}$. To this end, we will consider the case with $\gamma(x)\geq \gamma_0>0$ for a.e.\ $x\in\Omega$, i.e., a degenerate parabolic equation. In this setting, the variable change $\hat{u}(t)=\mathrm{e}^{-qt}u(t)$ gives the ``shifted'' equation
\begin{equation*}
\hat{\mathcal{F}}\hat{u}=\mathrm{d}/\mathrm{d} t\mathcal{M} \hat{u}+\hat{\mathcal{A}}\hat{u}+\hat{f}=0\quad\text{in }\mathcal{V}^*,
\end{equation*}
with $\hat{f}(t)=\mathrm{e}^{-qt}f(t)$ and
\begin{equation*}
\langle \hat{\mathcal{A}}\hat{u},v\rangle_{\mathcal{V}^*\times \mathcal{V}}=
\int_0^T\int_{\Omega} \mathrm{e}^{-qt}\alpha(t,\mathrm{e}^{qt}\nabla \hat{u})\cdot \nabla v+\bigl(\mathrm{e}^{-qt}\beta(t,\mathrm{e}^{qt}\hat{u})+q\gamma \hat{u}\bigr)v\,\mathrm{d}x\mathrm{d}t.
\end{equation*}
Note that~\cref{ass:parabolic} still holds for the operator above. This implies that $\hat{\mathcal{F}}\hat{u}=0$ has a unique solution. Furthermore, consider the decompositions, with $u,v\in \mathcal{V}_\ell$,
\begin{equation*}
\langle \hat{\mathcal{A}}_\ell u,v\rangle_{\mathcal{V}_\ell^*\times \mathcal{V}_\ell}=
\int_0^T\int_{\Omega_\ell} a_\ell\mathrm{e}^{-tq}\alpha(t,\mathrm{e}^{tq}\nabla u)\cdot \nabla v+\bigl(b_\ell\mathrm{e}^{-tq}\beta(t,\mathrm{e}^{tq}u)+\gamma u\bigr) v\,\mathrm{d}x\mathrm{d}t.
\end{equation*}
The same proof as for~\cref{lem:Aell} gives that $(V_\ell,H_\ell,M_\ell,\hat{\mathcal{A}}_\ell)$ is proper. Here, $\hat{\mathcal{A}}_\ell$ is $\hat{k}_\ell$-monotone with
\begin{equation*}
\hat{k}_\ell u=c\|\nabla u\|_{\mathcal{V}_\ell}^p+\gamma_0\|u\|^2_{\mathcal{H}}.
\end{equation*}
Applying~\cref{eq:AS} to this new decomposition will be referred to as the \emph{shifted additive splitting}. \cref{thm:convAS} then gives the following convergence result.
\begin{corollary}
Consider the shifted additive splitting approximation $\{\hat{u}^n\}_{n=1}^N$ of the solution~$u$ to the degenerate parabolic problem~$\mathcal{F}u=0$. If~\cref{ass:parabolic,ass:regularity,ass:rhs} hold and $\gamma(x)\geq \gamma_0>0$ for a.e.\ $x\in\Omega$, then
\begin{equation*}
\lim_{N\to\infty}\|\mathrm{e}^{qt}\hat{u}^N-u\|_{\mathcal{H}}=0
\end{equation*}
for the parameter choice $s=C\sqrt{N}$.
\end{corollary}

\bibliographystyle{plain}
\bibliography{references}

@article {Bardos,
    AUTHOR = {Bardos, Claude and Brezis, Ha\"im},
     TITLE = {Sur une classe de probl\`emes d'\'evolution non lin\'eaires},
   JOURNAL = {J. Differential Equations},
  FJOURNAL = {Journal of Differential Equations},
    VOLUME = {6},
      YEAR = {1969},
     PAGES = {345--394},
      ISSN = {0022-0396,1090-2732},
   MRCLASS = {47.65 (35.00)},
  MRNUMBER = {242020},
MRREVIEWER = {J.\ A.\ Goldstein},
       DOI = {10.1016/0022-0396(69)90023-0},
       URL = {https://doi.org/10.1016/0022-0396(69)90023-0},
}

@book {Show,
    AUTHOR = {Showalter, Ralph E.\ },
     TITLE = {Monotone operators in {B}anach space and nonlinear partial
              differential equations},
    SERIES = {Mathematical Surveys and Monographs},
    VOLUME = {49},
 PUBLISHER = {American Mathematical Society, Providence, RI},
      YEAR = {1997},
     PAGES = {xiv+278},
      ISBN = {0-8218-0500-2},
   MRCLASS = {47H15 (34G20 35J60 35K55 47H05 47N20)},
  MRNUMBER = {1422252},
MRREVIEWER = {Ioan\ I.\ Vrabie},
       DOI = {10.1090/surv/049},
       URL = {https://doi.org/10.1090/surv/049},
}

@book {Adams,
    AUTHOR = {Adams, Robert A. and Fournier, John J. F.},
     TITLE = {Sobolev spaces},
    SERIES = {Pure and Applied Mathematics (Amsterdam)},
    VOLUME = {140},
   EDITION = {Second},
 PUBLISHER = {Elsevier/Academic Press, Amsterdam},
      YEAR = {2003},
     PAGES = {xiv+305},
      ISBN = {0-12-044143-8},
   MRCLASS = {46E35 (46-01 46-02 46B70 46Exx)},
  MRNUMBER = {2424078},
}

@article {lionsmercier,
    AUTHOR = {Lions, Pierre-Louis and Mercier, Bertrand},
     TITLE = {Splitting algorithms for the sum of two nonlinear operators},
   JOURNAL = {SIAM J. Numer. Anal.},
  FJOURNAL = {SIAM Journal on Numerical Analysis},
    VOLUME = {16},
      YEAR = {1979},
    NUMBER = {6},
     PAGES = {964--979},
      ISSN = {0036-1429},
   MRCLASS = {47H17 (35K55 35R35 49D25 65J15)},
  MRNUMBER = {551319},
MRREVIEWER = {D. Pascali},
       DOI = {10.1137/0716071},
       URL = {https://doi.org/10.1137/0716071},
}

@book {roubicek,
    AUTHOR = {Roub\'{\i}\v{c}ek, Tom\'{a}\v{s}},
     TITLE = {Nonlinear partial differential equations with applications},
    SERIES = {International Series of Numerical Mathematics},
    VOLUME = {153},
   EDITION = {Second},
 PUBLISHER = {Birkh\"{a}user/Springer Basel AG, Basel},
      YEAR = {2013},
     PAGES = {xx+476},
      ISBN = {978-3-0348-0512-4; 978-3-0348-0513-1},
   MRCLASS = {35-02 (35J60 35K55 35Q30)},
  MRNUMBER = {3014456},
       DOI = {10.1007/978-3-0348-0513-1},
       URL = {https://doi.org/10.1007/978-3-0348-0513-1},
}

@article {Eisenmann,
    AUTHOR = {Eisenmann, Monika and Hansen, Eskil},
     TITLE = {Convergence analysis of domain decomposition based time
              integrators for degenerate parabolic equations},
   JOURNAL = {Numer. Math.},
  FJOURNAL = {Numerische Mathematik},
    VOLUME = {140},
      YEAR = {2018},
    NUMBER = {4},
     PAGES = {913--938},
      ISSN = {0029-599X,0945-3245},
   MRCLASS = {65M55 (35K65 65J08 65M12)},
  MRNUMBER = {3864705},
MRREVIEWER = {Barry\ Lee},
       DOI = {10.1007/s00211-018-0985-z},
       URL = {https://doi.org/10.1007/s00211-018-0985-z},
}

@book {Kufner,
    AUTHOR = {Kufner, Alois and John, Old\v{r}ich and Fu\v{c}\'{\i}k, Svatopluk},
     TITLE = {Function spaces},
     PUBLISHER = {Noordhoff International Publishing, Leyden; Academia, Prague},
      YEAR = {1977},
     PAGES = {xv+454},
      ISBN = {90-286-0015-9},
   MRCLASS = {46EXX},
  MRNUMBER = {0482102},
MRREVIEWER = {Richard Bagby},
}

@book {zeidler,
    AUTHOR = {Zeidler, Eberhard},
     TITLE = {Nonlinear functional analysis and its applications. {II}/{B}},
 PUBLISHER = {Springer-Verlag, New York},
      YEAR = {1990},
     PAGES = {i--xvi and 469--1202},
      ISBN = {0-387-97167-X},
   MRCLASS = {47-02 (35-01 35J60 47Hxx 58-01 65Jxx)},
  MRNUMBER = {1033498},
MRREVIEWER = {Jean Mawhin},
       DOI = {10.1007/978-1-4612-0985-0},
       URL = {https://doi.org/10.1007/978-1-4612-0985-0},
}

@article {Temam,
    AUTHOR = {Temam, Roger},
     TITLE = {Remarks on the approximation of some nonlinear elliptic
              equations},
   JOURNAL = {J. Comput. System Sci.},
  FJOURNAL = {Journal of Computer and System Sciences},
    VOLUME = {4},
      YEAR = {1970},
     PAGES = {250--259},
      ISSN = {0022-0000},
   MRCLASS = {65.66},
  MRNUMBER = {272211},
MRREVIEWER = {R.\ S.\ Varga},
       DOI = {10.1016/S0022-0000(70)80023-X},
       URL = {https://doi.org/10.1016/S0022-0000(70)80023-X},
}

@book {quarteroni,
    AUTHOR = {Quarteroni, Alfio and Valli, Alberto},
     TITLE = {Domain decomposition methods for partial differential
              equations},
    SERIES = {Numerical Mathematics and Scientific Computation},
      NOTE = {Oxford Science Publications},
 PUBLISHER = {The Clarendon Press, Oxford University Press, New York},
      YEAR = {1999},
     PAGES = {xvi+360},
      ISBN = {0-19-850178-1},
   MRCLASS = {65-02 (35J25 35Q30 35Q35 65M55 65N55)},
  MRNUMBER = {1857663},
MRREVIEWER = {Weimin Han},
}

@book {Widlund,
    AUTHOR = {Toselli, Andrea and Widlund, Olof},
     TITLE = {Domain decomposition methods---algorithms and theory},
    SERIES = {Springer Series in Computational Mathematics},
    VOLUME = {34},
 PUBLISHER = {Springer-Verlag, Berlin},
      YEAR = {2005},
     PAGES = {xvi+450},
      ISBN = {3-540-20696-5},
   MRCLASS = {65-02 (65N55 74S05 76M10)},
  MRNUMBER = {2104179},
MRREVIEWER = {R\'{e}mi Vaillancourt},
       DOI = {10.1007/b137868},
       URL = {https://doi-org.ludwig.lub.lu.se/10.1007/b137868},
}

@incollection {vabi,
    AUTHOR = {Vabishchevich, Petr and Zakharov, Petr},
     TITLE = {Domain decomposition scheme for first-order evolution
              equations with nonselfadjoint operators},
 BOOKTITLE = {Numerical solution of partial differential equations: theory,
              algorithms, and their applications},
    SERIES = {Springer Proc. Math. Stat.},
    VOLUME = {45},
     PAGES = {279--302},
 PUBLISHER = {Springer, New York},
      YEAR = {2013},
      ISBN = {978-1-4614-7172-1; 978-1-4614-7171-4},
   MRCLASS = {65M06 (65M55)},
  MRNUMBER = {3107037},
       DOI = {10.1007/978-1-4614-7172-1\_14},
       URL = {https://doi.org/10.1007/978-1-4614-7172-1_14},
}

@article {henningsson,
    AUTHOR = {Hansen, Eskil and Henningsson, Erik},
     TITLE = {Additive domain decomposition operator
              splittings---convergence analyses in a dissipative framework},
   JOURNAL = {IMA J. Numer. Anal.},
  FJOURNAL = {IMA Journal of Numerical Analysis},
    VOLUME = {37},
      YEAR = {2017},
    NUMBER = {3},
     PAGES = {1496--1519},
      ISSN = {0272-4979,1464-3642},
   MRCLASS = {65M55},
  MRNUMBER = {3671503},
       DOI = {10.1093/imanum/drw043},
       URL = {https://doi.org/10.1093/imanum/drw043},
}

@article {gander25,
    AUTHOR = {Gander, Martin J. and Wu, Shu-Lin and Zhou, Tao},
     TITLE = {Time parallelization for hyperbolic and parabolic problems},
   JOURNAL = {Acta Numer.},
  FJOURNAL = {Acta Numerica},
    VOLUME = {34},
      YEAR = {2025},
     PAGES = {385--489},
      ISSN = {0962-4929,1474-0508},
   MRCLASS = {65M55 (65M06 65M12 65M15 65Y05)},
  MRNUMBER = {4926314},
       DOI = {10.1017/S0962492924000072},
       URL = {https://doi.org/10.1017/S0962492924000072},
}

@incollection {steinbach19,
    AUTHOR = {Steinbach, Olaf and Yang, Huidong},
     TITLE = {Space-time finite element methods for parabolic evolution
              equations: discretization, a posteriori error estimation,
              adaptivity and solution},
 BOOKTITLE = {Space-time methods---applications to partial differential
              equations},
    SERIES = {Radon Ser. Comput. Appl. Math.},
    VOLUME = {25},
     PAGES = {207--248},
 PUBLISHER = {De Gruyter, Berlin},
      YEAR = {2019},
       DOI = {10.1515/9783110548488-007},
       URL = {https://doi.org/10.1515/9783110548488-007},
}

@article {japhet20,
    AUTHOR = {Ahmed, Elyes and Japhet, Caroline and Kern, Michel},
     TITLE = {Space-time domain decomposition for two-phase flow between
              different rock types},
   JOURNAL = {Comput. Methods Appl. Mech. Engrg.},
  FJOURNAL = {Computer Methods in Applied Mechanics and Engineering},
    VOLUME = {371},
      YEAR = {2020},
     PAGES = {113294, 30},
       DOI = {10.1016/j.cma.2020.113294},
       URL = {https://doi.org/10.1016/j.cma.2020.113294},
}

@article {gander07,
    AUTHOR = {Gander, Martin J. and Halpern, Laurence},
     TITLE = {Optimized {S}chwarz waveform relaxation methods for advection
              reaction diffusion problems},
   JOURNAL = {SIAM J. Numer. Anal.},
  FJOURNAL = {SIAM Journal on Numerical Analysis},
    VOLUME = {45},
      YEAR = {2007},
    NUMBER = {2},
     PAGES = {666--697},
       DOI = {10.1137/050642137},
       URL = {https://doi.org/10.1137/050642137},
}

@article {kwok21,
    AUTHOR = {Gander, Martin J. and Kwok, Felix and Mandal, Bankim C.},
     TITLE = {Dirichlet-{N}eumann waveform relaxation methods for parabolic
              and hyperbolic problems in multiple subdomains},
   JOURNAL = {BIT},
  FJOURNAL = {BIT. Numerical Mathematics},
    VOLUME = {61},
      YEAR = {2021},
    NUMBER = {1},
     PAGES = {173--207},
       DOI = {10.1007/s10543-020-00823-2},
       URL = {https://doi.org/10.1007/s10543-020-00823-2},
}

@article {halp12,
    AUTHOR = {Halpern, Laurence and Japhet, Caroline and Szeftel, J\'{e}r\'{e}mie},
     TITLE = {Optimized {S}chwarz waveform relaxation and discontinuous
              {G}alerkin time stepping for heterogeneous problems},
   JOURNAL = {SIAM J. Numer. Anal.},
  FJOURNAL = {SIAM Journal on Numerical Analysis},
    VOLUME = {50},
      YEAR = {2012},
    NUMBER = {5},
     PAGES = {2588--2611},
      ISSN = {0036-1429},
   MRCLASS = {65M60 (65M55)},
  MRNUMBER = {3022233},
MRREVIEWER = {Istv\'{a}n Farag\'{o}},
       DOI = {10.1137/120865033},
       URL = {https://doi-org.ludwig.lub.lu.se/10.1137/120865033},
}

@article {tai02,
    AUTHOR = {Tai, Xue-Cheng and Xu, Jinchao},
     TITLE = {Global and uniform convergence of subspace correction methods
              for some convex optimization problems},
   JOURNAL = {Math. Comp.},
  FJOURNAL = {Mathematics of Computation},
    VOLUME = {71},
      YEAR = {2002},
    NUMBER = {237},
     PAGES = {105--124},
      ISSN = {0025-5718},
   MRCLASS = {65K10 (65J15 65N22 65N55)},
  MRNUMBER = {1862990},
       DOI = {10.1090/S0025-5718-01-01311-4},
       URL = {https://doi.org/10.1090/S0025-5718-01-01311-4},
}

@article {engstrom22,
    AUTHOR = {Engstr\"om, Emil and Hansen, Eskil},
     TITLE = {Convergence analysis of the nonoverlapping {R}obin-{R}obin
              method for nonlinear elliptic equations},
   JOURNAL = {SIAM J. Numer. Anal.},
  FJOURNAL = {SIAM Journal on Numerical Analysis},
    VOLUME = {60},
      YEAR = {2022},
    NUMBER = {2},
     PAGES = {585--605},
      ISSN = {0036-1429,1095-7170},
   MRCLASS = {65N55 (35J70 47N20 65N12 65N30)},
  MRNUMBER = {4396361},
       DOI = {10.1137/21M1414942},
       URL = {https://doi.org/10.1137/21M1414942},
}

@article {engstrom24,
    AUTHOR = {Engstr\"om, Emil and Hansen, Eskil},
     TITLE = {Linearly convergent nonoverlapping domain decomposition
              methods for quasilinear parabolic equations},
   JOURNAL = {BIT},
  FJOURNAL = {BIT. Numerical Mathematics},
    VOLUME = {64},
      YEAR = {2024},
    NUMBER = {4},
     PAGES = {Paper No. 37, 37},
      ISSN = {0006-3835,1572-9125},
   MRCLASS = {65M55 (35K20 35K59 65J08)},
  MRNUMBER = {4800692},
MRREVIEWER = {R\"udiger\ Verf\"urth},
       DOI = {10.1007/s10543-024-01038-5},
       URL = {https://doi.org/10.1007/s10543-024-01038-5},
}

\end{document}